\documentclass[12pt]{amsart}
\usepackage{amssymb}
\usepackage{stmaryrd}
\usepackage{eucal}
\usepackage{longtable}
\usepackage{hyperref}
\usepackage[nobysame]{amsrefs}
\usepackage[total={6.25truein,9truein},centering]{geometry}

\allowdisplaybreaks

\newtheorem{theorem}{Theorem}[section]
\newtheorem{lemma}[theorem]{Lemma}
\newtheorem{proposition}[theorem]{Proposition}
\newtheorem{corollary}[theorem]{Corollary}

\theoremstyle{definition}

\theoremstyle{remark}
\newtheorem{example}[theorem]{Example}
\newtheorem{remark}[theorem]{Remark}

\newcommand{\FF}{\mathbb{F}}
\newcommand{\ZZ}{\mathbb{Z}}
\newcommand{\QQ}{\mathbb{Q}}
\newcommand{\RR}{\mathbb{R}}

\newcommand{\GG}{\mathbb{G}}
\newcommand{\HH}{\mathbb{H}}

\newcommand{\CC}{\mathbb{C}}

\newcommand{\cE}{\mathcal{E}}
\newcommand{\cF}{\mathcal{F}}
\newcommand{\cG}{\mathcal{G}}
\newcommand{\cL}{\mathcal{L}}

\newcommand{\fa}{\mathfrak{a}}
\newcommand{\fn}{\mathfrak{n}}

\DeclareMathOperator{\ord}{ord}

\DeclareMathOperator{\im}{Im}
\DeclareMathOperator{\re}{Re}

\newcommand{\ogamma}{\overline{\gamma}}

\newcommand{\opsi}{\mkern2.5mu\overline{\mkern-2.5mu \psi}}

\newcommand{\oQ}{\overline{Q}}

\newcommand{\tor}{\mathrm{tor}}

\newcommand{\Gahat}{\widehat{\GG}_{\mathrm{a}}}

\newcommand{\power}[2]{{#1 [\![ #2 ]\!]}}
\newcommand{\laurent}[2]{{#1 (\!( #2 )\!)}}

\newcommand{\assign}{\mathrel{\vcenter{\baselineskip0.5ex \lineskiplimit0pt
                     \hbox{\scriptsize.}\hbox{\scriptsize.}}}%
                     =}
\newcommand{\rassign}{=%
                     \mathrel{\vcenter{\baselineskip0.5ex \lineskiplimit0pt
                     \hbox{\scriptsize.}\hbox{\scriptsize.}}}%
                     }

\begin{document}

\title{A note on log-algebraicity on elliptic curves}

%    Information for first author
\author{Wei-Cheng Huang}
\address{Department of Mathematics, Texas A{\&}M University, College Station, TX 77843, U.S.A.}
\email[W.-C.~Huang]{wchuang@tamu.edu}

%    Information for second author
\author{Matthew Papanikolas}
%\address{Department of Mathematics, Texas A{\&}M University, College Station, TX 77843, U.S.A.}
\email[M.~Papanikolas]{papanikolas@tamu.edu}


\date{April 8, 2022}

\begin{abstract}
We analyze log-algebraic power series identities for formal groups of elliptic curves over $\mathbb{Q}$ which arise from modular parametrizations. We further investigate applications to special values of elliptic curve $L$-functions.
\end{abstract}


%\thanks{\emph{Availability of data.} All supporting data for this project are included in this article.}

\maketitle

\section{Introduction} \label{S:Intro}

The notion of log-algebraicity was termed by Anderson~\cite{And94}, \cite{And96}, to describe certain power series identities in the context of exponential functions of Drinfeld modules and twisted harmonic series over global function fields. The basic motivating example from characteristic~$0$ is the familiar formal power series identity for the multiplicative group,
\begin{equation}
\exp \Biggl( -\sum_{n=1}^{\infty} \frac{t^n}{n} \Biggr) = 1-t.
\end{equation}
If we take $\beta = \sum_{k=-d}^d m_k u^k \in \ZZ[u,u^{-1}]$, then only slightly more complicated is that
\begin{equation}
\exp \Biggl( -\sum_{n=1}^{\infty} \frac{\beta(u^n)}{n}\, t^n \Biggr) = \prod_{k=-d}^d \Bigl( 1 - u^k t \Bigr)^{m_k}.
\end{equation}
These inner harmonic sums are thus ``log-algebraic,'' as one obtains polynomial, rational, or algebraic power series upon exponentiation. Specializations of these identities can be used to recover classical formulas for Dirichlet $L$-functions at $s=1$, as in~\cite{Washington}*{Thm.~4.9}.

Anderson extended these types of identities to sign-normalized Drinfeld modules over function fields in positive characteristic, based on special identities for Carlitz-Goss zeta values due to Thakur~\cite{Thakur92}. For example, Anderson~\cite{And96}*{Thm.~3} showed for the Carlitz module $C$ over $A=\FF_q[\theta]$, $K=\FF_q(\theta)$, that in $\power{K[u]}{t}$ we have
\begin{equation}
\exp_C \Biggl( \sum_{a \in A_+} \frac{\beta(C_a(u))}{a}\, t^{q^{\deg a}} \Biggr) \in A[u,t],
\end{equation}
where $\exp_C$ is the Carlitz exponential, $A_+$ denotes the monic elements of $A$, $\beta \in A[u]$ is fixed, and $C_a(u) \in A[u]$ represents multiplication by~$a$ on $C$.  Anderson used these identities to express special values of Goss $L$-series of Dirichlet type in terms of Carlitz logarithms of special points, which themselves arise from a theory of circular units for~$C$. Log-algebraic identities have been widely studied in function field arithmetic in recent years, extending Anderson's results to analogues of Stark units and abelian $L$-series, Drinfeld modules of arbitrary ranks, and certain Anderson $t$-modules (e.g., see \cite{AndThak90}--\cite{AnglesTavares17}, \cite{CEP18}, \cite{GreenNgoDac20b}, \cite{GreenP18}, \cite{P22}, \cite{Thakur}*{Ch.~8}). These identities are also closely connected to Taelman's work on special values of Goss $L$-series for Drinfeld modules~\cite{Taelman12}. \nocite{AnglesNgoDacTavares20} \nocite{AnglesPellarinTavares16}

Although log-algebraic identities were constructed in the theory of function fields, the purpose of the present note is to investigate how they occur in characteristic~$0$ on other algebraic groups, particularly on the formal groups of elliptic curves over~$\QQ$. Unlike for Drinfeld modules and Anderson $t$-modules over~$A$ or for the multiplicative group, where we exponentiate twisted harmonic power series, we find that log-algebraic formulas for an elliptic curve over~$\QQ$ arise most naturally through the curve's modular parametrization.

Our main results in these directions are in \S\ref{S:Formal} (see Corollary~\ref{C:logalg1} and Theorem~\ref{T:Main}). For example, we show for an elliptic curve $E/\QQ$ and $\beta = \sum_{k=0}^d m_k u^k \in \ZZ[u]$, that in $\power{\QQ[u]}{t}$,
\begin{equation}
\exp_{\cE} \Biggl( \sum_{n=1}^{\infty} \frac{a_n \beta(u^n)}{n}\,t^n \Biggr) = \sideset{}{_\cE}\sum_{k=0}^d [m_k]_{\cE}\bigl( \Phi(u^k t) \bigr).
\end{equation}
Here $\exp_{\cE}(t)$ denotes the exponential on the formal group of $E$, the sequence $\{a_n\}$ provides the Fourier coefficients of the newform $f \in S_2(\Gamma_0(N))$ attached to $E$, $\Phi(t) \in \power{\QQ}{t}$ is induced by the modular parametrization, and the sum ``$\sum_{\cE}$'' on the right is taken in the formal group~$\cE$. The reason to term this identity as ``log-algebraic'' is that the series $\Phi(t)$ formally represents the algebraic map $X_0(N) \to E$, and so the right-hand side represents a formal sum of algebraic points on~$E$. Thus the interior sum on the left is a formal logarithm of this sum of algebraic points. This identity is also closely related to the formal group of Honda~\cite{Honda68}, \cite{Honda70}, obtained from~$f$.

We present examples of applications to $L$-functions of elliptic curves in \S\ref{S:Examples}, by demonstrating how Corollary~\ref{C:logalg1} and Theorem~\ref{T:Main} can be used to determine exactly the values of $L(E,1)$ and $L(E,\chi,1)$ for a Dirichlet character $\chi$. For example, if $E_0 = X_0(11)$ and $\psi$ is a given cubic Dirichlet character modulo~$7$, we show in Example~\ref{Ex:three} that
\[
L(E_0/\QQ,\psi,1) = \frac{5}{14}(1+\sqrt{-3}) g(\psi) \Omega,
\]
where $g(\psi)$ is a Gauss sum and $\Omega$ is the positive real period of $E_0$.
To be sure, these special value formulas can be obtained using previous methods of modular symbols that are not far from our considerations, e.g., see~\cite{Cremona}*{\S 2.8--2.12} and~\cite{Shimura77}, but one underlying goal of this note is to investigate how special $L$-values are interpolated by power series identities. 

\section{Formal groups and log-algebraic identities} \label{S:Formal}

Suppose we have an elliptic curve,
\[
E : y^2 = x^3 - \frac{g_2}{4} x - \frac{g_3}{4}, \quad g_2,\ g_3 \in \CC,
\]
and let $\Lambda \subseteq \CC$ be its associated lattice so that $g_i = g_i(\Lambda)$ for $i=2$, $3$. Let $\wp(z) = \wp_{\Lambda}(z)$ be its associated Weierstrass $\wp$-function so that $(\wp(z),\tfrac12 \wp'(z))$ represents a point on $E(\CC)$. Recall by~\cite{SilvermanAEC}*{Thm.~VI.3.5} that the Laurent series expansion of $\wp(z)$ at $z=0$ is
\[
  \wp(z) = \frac{1}{z^2} + \sum_{k=1}^{\infty} (2k+1) G_{2k+2}(\Lambda) z^{2k},
\]
where $G_{2k}(\Lambda)$ is the weight $2k$ Eisenstein series for $\Lambda$.

Let $\cE$ denote the formal group of $E$ over $\QQ[g_2,g_3]$, and by abuse of notation we let $\cE(t_1,t_2) \in \power{\QQ[g_2,g_3]}{t_1,t_2}$ be the power series defining this formal group law as in \cite{SilvermanAEC}*{Ch.~IV}. In particular we have the formal $x$- and $y$-coordinates in $\laurent{\QQ[g_2,g_3]}{t}$,
\begin{equation} \label{E:xtyt}
x(t) = \frac{1}{t^2} + \frac{g_2}{4}t^2 + \frac{g_3}{4}t^4 - \frac{g_2^2}{16}t^6 + \cdots, \quad
y(t) = -\frac{1}{t^3} - \frac{g_2}{4}t - \frac{g_3}{4}t^3 + \frac{g_2^2}{16}t^5 + \cdots,
\end{equation}
and formal invariant differential,
\[
\omega_E(t) = \biggl( 1 - \frac{g_2}{2} t^4 - \frac{3g_3}{4} t^6 + \frac{3g_2^2}{8}t^8 + \cdots \biggr)\,dt.
\]
We collect a couple results about formal groups. This first follows from~\cite{SilvermanAEC}*{\S VII.2}.

\begin{lemma} \label{L:formaliso}
The maps
\[
\alpha(t) \mapsto \bigl( x(\alpha(t)), y(\alpha(t)) \bigr), \quad
-\frac{x_0(t)}{y_0(t)} \mapsfrom  (x_0(t), y_0(t)),
\]
induce mutually inverse isomorphisms of abelian groups,
\[
\cE\bigl( t\power{\QQ(g_2,g_3)}{t} \bigr) \cong \bigl\{ (x_0(t),y_0(t)) \in E\bigl( \laurent{\QQ(g_2,g_3)}{t} \bigr) : \ord_t \bigl( x_0(t)/y_0(t) \bigr) \geqslant 1 \bigr\} \cup \{ O\}.
\]
\end{lemma}

This next follows from the fact that the formal exponential and logarithms for a formal group over a field of characteristic~$0$ are isomorphisms with the formal additive group~$\Gahat$ (see \cite{Hazewinkel}*{\S 5.4} or \cite{SilvermanAEC}*{\S IV.5})

\begin{lemma} \label{L:morphism}
Let $K$ be a field of characteristic~$0$, and let $\cF(t_1,t_2)$, $\cG(t_1,t_2) \in \power{K}{t_1,t_2}$ be formal groups over $K$. For a power series $m(t) \in \power{K}{t}$, the following are equivalent.
\begin{itemize}
\item[(a)] $m : \cF \to \cG$ is a morphism of formal groups,
\item[(b)] $m \circ \exp_{\cF}(t) = \exp_{\cG}(m'(0)t)$,
\item[(c)] $\log_{\cG} \circ\, m(t) = m'(0)\cdot \log_{\cF}(t)$,
\item[(d)] $\omega_{\cG} \circ m(t) = m'(0) \cdot \omega_{\cF}(t)$.
\end{itemize}
\end{lemma}

We then obtain the following result for the formal exponential of~$\cE$.

\begin{proposition} \label{P:ExpE}
Considering $\wp(z) \in \laurent{\QQ[g_2,g_3]}{z}$ as a formal Laurent series in~$z$,
\[
  \exp_{\cE}(z) = -\frac{2\wp(z)}{\wp'(z)} \in \power{\QQ[g_2,g_3]}{z}.
\]
\end{proposition}

\begin{proof}
The invariant differential $\omega_E$ on $E$ satisfies $\omega_E = dx/2y = d\wp(z)/\wp'(z) = dz$.  On the other hand, letting $t=-x/y$ be the formal parameter on $\cE$, the invariant differential formally satisfies
$\omega_{\cE}(t) = dx(t)/2y(t)$. By definition the formal logarithm $\log_{\cE}(t)$ satisfies
\[
\log_{\cE}(t) = \int \omega_{\cE}(t) = \int dz = z,
\]
and so
\[
\exp_{\cE}(z) = t = -\frac{x(z)}{y(z)} = -\frac{2\wp(z)}{\wp'(z)}.
\qedhere
\]
\end{proof}

Now let $E_0/\QQ$ be an elliptic curve of conductor~$N$,
\[
E_0 : y^2 + e_1 xy + e_3 y = x^3 + e_2 x^2 + e_4 x + e_6, \quad e_i \in \ZZ,
\]
which is modular by~\cite{BCDT01}, \cite{TW95}, \cite{Wiles95}. Assume further that the modular parametrization $\mu : X_0(N) \to E_0$ is optimal, i.e., $E_0$ is a strong Weil curve. Let $f \in S_2(\Gamma_0(N))$ be the unique normalized newform associated to $E_0$ with Fourier expansion
\[
f(\tau) = \sum_{n=1}^{\infty} a_n q^n, \quad q = e^{2\pi i \tau},\ a_n \in \ZZ,\ a_1 = 1.
\]
Then the pullback of the invariant differential $\omega_{E_0}$ satisfies
\[
  \mu^{*} \omega_{E_0} = c \cdot 2\pi i f(\tau)\,d\tau =
  c\cdot \sum_{n=1}^{\infty} a_n q^{n-1}\,dq,
\]
where $c$ is the Manin constant, which henceforth we assume to be~$1$. Replacing $q$ by a formal parameter $t$, we set
\begin{equation} \label{E:lambda}
\lambda(t) \assign \int \mu^{*} \omega_{E_0} = \sum_{n=1}^{\infty} \frac{a_n}{n} t^n,
\end{equation}
and recall a theorem of Honda.

\begin{theorem}[{Honda~\cite{Honda68}*{Thm.~5}, \cite{Honda70}}] \label{T:Honda}
There is a formal group $\cL$ over $\ZZ$ so that
\[
  \log_{\cL}(t) = \lambda(t),
\]
and in particular, $\cL(t_1, t_2) = \lambda^{-1} ( \lambda(t_1)+\lambda(t_2)) \in \power{\ZZ}{t_1,t_2}$. Furthermore, $\cL$ is strongly isomorphic over $\ZZ$ to the formal group of $E_0$.
\end{theorem}

We also recall that~\eqref{E:lambda} is closely related to the Eichler integral
\[
  2\pi i \int_{z}^{i\infty} f(\tau)\,d\tau = -\sum_{n=1}^{\infty} \frac{a_n}{n} e^{2\pi i nz},
\]
for $z \in \CC$ with $\im(z) > 0$, and as such $\lambda(t)$ is a formal Eichler integral in the sense of~\cite{BGKO13}.

Defining invariants $b_i$, $c_i \in \QQ$ as in \cite{SilvermanAEC}*{\S III.1}, we change coordinates on $E_0$ by
\begin{equation} \label{E:coordchange}
x \leftarrow x - \frac{b_2}{12}, \quad y \leftarrow y - \frac{e_1}{2} x + \frac{e_1 b_2}{24} - \frac{e_3}{2},
\end{equation}
and obtain the $\QQ$-isomorphism $\psi : E_0 \to E$, where
\[
E : y^2 = x^3 - \frac{c_4}{48} x - \frac{c_6}{864}.
\]
Setting $g_2 \assign c_4/12$ and $g_3 \assign c_6/216$, we obtain $\Lambda$, $\wp(z)=\wp_\Lambda(z)$, $\log_{\cE}(z)$, and $\exp_{\cE}(z)$, all associated to $E$, as in the beginning of this section. Thus we can define $\phi \assign \psi \circ \mu$,
\[
\phi : X_0(N) \xrightarrow{\mu} E_0 \xrightarrow{\psi} E,
\]
from which we see that
\[
\phi^* \omega_E =\mu^* \omega_{E_0} = \sum_{n=1}^{\infty} a_n q^{n-1}\,dq = d\lambda(q).
\]
Moreover, we can set $X(q) \assign \phi^*(x)$ and $Y(q) \assign \phi^*(y)$ in $\QQ(X_0(N))$. Then as $X(q)$, $Y(q)$ satisfy the defining equation for~$E$ and
\begin{equation} \label{E:XYandf}
q \cdot dX/dq = 2Y\cdot f,
\end{equation}
it follows recursively (see \cite{Cremona}, \cite{GriffinHales20}*{\S 3}) that $X$, $Y$ have Laurent series expansions in~$q$,
\begin{equation} \label{E:XYq}
X(q) = \frac{1}{q^2} -\frac{a_2}{q} + \frac{3a_2^2}{4} - \frac{2a_3}{3} + \cdots, \quad Y(q) = -\frac{1}{q^3} + \frac{3a_2}{2q^2} - \frac{3a_2^2-2a_3}{2q} + \cdots.
\end{equation}
We let $X(t)$, $Y(t) \in \laurent{\QQ}{t}$ be the formal series obtained by replacing $q$ with $t$, and we set
\begin{equation} \label{E:Phidef}
\Phi(t) \assign -\frac{X(t)}{Y(t)} = t + \frac{a_2}{2}t^2 + \frac{a_3}{3}t^3 + \frac{a_4}{4} t^4 + \biggl( \frac{c_4}{120} + \frac{a_5}{5} \biggr) t^5 + \cdots \in \power{\QQ}{t}.
\end{equation}
The following proposition underlies our log-algebraic identities.

\begin{proposition} \label{P:Phiiso}
The power series $\Phi(t) \in \power{\QQ}{t}$ is an isomorphism $\Phi: \cL \to \cE$ of formal groups over $\QQ$.
\end{proposition}

\begin{proof}
Combining Lemma~\ref{L:formaliso} and \eqref{E:Phidef}, we see that
\begin{equation} \label{E:XYPhi}
X(t) = x(\Phi(t)), \quad Y(t) = y(\Phi(t)).
\end{equation}
It then follows that
\[
\omega_E \circ \Phi = \frac{d\bigl(x \circ \Phi(t)\bigr)}{2y \circ \Phi(t)} = \frac{dX(t)}{2Y(t)} = \sum_{n=1}^{\infty} a_n t^{n-1}\,dt,
\]
where the last equality follows from~\eqref{E:XYandf}. Then by Theorem~\ref{T:Honda} we conclude that $\omega_E \circ \Phi = \Phi'(0)\omega_{\cL}$ (since $\Phi'(0)=1$), and the result follows from Lemma~\ref{L:morphism}.
\end{proof}

We then obtain the following log-algebraic identity on power series in $\power{\QQ}{t}$ together with a specialization relating it to a special $L$-value. This expression is ``log-algebraic'' in that $\Phi(t)$ formally represents the algebraic map $\phi: X_0(N) \to E$.

\begin{corollary} \label{C:logalg1}
For $E_0/\QQ$ a strong Weil curve of conductor $N$, let $E/\QQ$, $\exp_{\cE}(t)$, and $\Phi(t)$ be chosen as above.
\begin{itemize}
\item[(a)] We have the identity of formal power series in $\power{\QQ}{t}$,
\[
\exp_{\cE} \Biggl( \sum_{n=1}^{\infty} \frac{a_n}{n} t^n \Biggr) = \Phi(t).
\]
\item[(b)] Suppose that the sign of the functional equation of $L(E_0/\QQ,s) = L(f,s)$ is $\varepsilon = {+1}$. Then
\[
\wp\Bigl( \tfrac12 L(E_0/\QQ,1) \Bigr) = X\bigl( e^{-2\pi/\sqrt{N}} \bigr),
\]
when both sides converge.
\end{itemize}
\end{corollary}

\begin{proof}
For (a) we combine Theorem~\ref{T:Honda} and Proposition~\ref{P:Phiiso} (again using $\Phi'(0)=1$). For~(b), in general if $\varepsilon = \pm 1$ is the sign of the functional equation, then by~\cite{Cremona}*{Prop.~2.11.1} we have the rapidly converging formula
\begin{equation} \label{E:Lrapid}
L(E_0/\QQ,1) = (1 + \varepsilon) \sum_{n=1}^{\infty} \frac{a_n}{n} e^{-2\pi n/\sqrt{N}}.
\end{equation}
When $\varepsilon={+1}$, we could obtain $\exp_{\cE} ( \tfrac12 L(E_0/\QQ,1)) = \Phi( \exp(-2\pi/\sqrt{N}))$ by specializing into part~(a), and this can converge for some elliptic curves $E_0$ and $E$ (see Example~\ref{Ex:one}), but it runs into issues near zeros of $Y(t)$ since $\Phi(t)= -X(t)/Y(t)$. Instead, using Lemma~\ref{L:formaliso}, we see that we have identities in $\laurent{\QQ}{t}$: $x( \exp_{\cE}(t)) = \wp(t)$ and $y(\exp_{\cE}(t)) = \tfrac12 \wp'(t)$.
Combining part (a) with~\eqref{E:XYPhi}, we obtain formal identities in $\laurent{\QQ}{t}$,
\begin{equation} \label{E:wp}
\wp \Biggl( \sum_{n=1}^{\infty} \frac{a_n}{n}t^n \Biggr) = X(t), \quad
\frac12 \wp' \Biggl(  \sum_{n=1}^{\infty} \frac{a_n}{n}t^n \Biggr) = Y(t).
\end{equation}
Our desired identity follows from \eqref{E:Lrapid} by substituting $t=e^{-2\pi/\sqrt{N}}$. We note that this is the same as letting $t=e^{2\pi i\tau}$ with $\tau=i/\sqrt{N}$ from the upper half-plane.
\end{proof}

\begin{remark} \label{R:wp}
Even when convergence is an issue in Corollary~\ref{C:logalg1}(b), one finds that
\[
\wp\Bigl( \tfrac12 L(E_0/\QQ,1) \Bigr) = x\bigl(\phi(i/\sqrt{N}) \bigr),
\]
where as usual $\wp$ is extended to a meromorphic function on~$\CC$.
\end{remark}

Our main result constitutes the following power series identities.

\begin{theorem} \label{T:Main}
For $E_0/\QQ$ a strong Weil curve of conductor $N$, let $E/\QQ$, $\wp$, $\cE/\power{\QQ}{t_1,t_2}$, $\exp_{\cE}(t)$, and $\Phi(t)$ be chosen as above. Let $\beta = \sum_{k=0}^d m_k u^k \in \ZZ[u]$.
\begin{itemize}
\item[(a)] We have the identity of formal power series in $\power{\QQ[u]}{t}$,
\[
\exp_{\cE} \Biggl( \sum_{n=1}^{\infty} \frac{a_n \beta(u^n)}{n}\,t^n \Biggr) = \sideset{}{_\cE}\sum_{k=0}^d [m_k]_{\cE}\bigl( \Phi(u^k t) \bigr),
\]
where $\sum_{\cE}$ indicates that the sum is taken with respect to the formal group law~$\cE$.
\item[(b)] Let $P(t) \assign (X(t),Y(t)) \in E(\laurent{\QQ}{t})$. Then in $\laurent{\QQ(u)}{t}$,
\[
\wp \Biggl(  \sum_{n=1}^{\infty} \frac{a_n \beta(u^n)}{n}\,t^n \Biggr) = x \Biggl( \sideset{}{_E}\sum_{k=0}^d [m_k]_E \bigl( P(u^kt) \bigr) \Biggr),
\]
where $\sum_E$ indicates that the sum is taken with respect to the group law on~$E$.
\end{itemize}
\end{theorem}

\begin{proof}
Part~(a) is a consequence of Corollary~\ref{C:logalg1}. We recall that since $\exp_{\cE} : \Gahat \to \cE$ is a morphism of formal groups, for $g$, $h \in t\cdot \power{\QQ[u]}{t}$, we have $\exp_{\cE}(g + h) = \cE(g,h)$ and likewise for $m \in \ZZ$, we have $\exp_{\cE}(mg) = [m]_{\cE}(\exp_{\cE}(g))$. We then observe that
\[
\sum_{n=1}^{\infty} \frac{a_n \beta(u^n)}{n} t^n = \sum_{k=0}^d m_k \sum_{n=1}^{\infty} \frac{a_n}{n} \bigl( u^k t \bigr)^n.
\]
By applying $\exp_{\cE}$ to both sides and using Corollary~\ref{C:logalg1}, we obtain~(a). For part~(b), we take the $x$-coordinate of both sides of~(a) and use~\eqref{E:wp}. If $g \in t\cdot \power{\QQ[u]}{t}$, then $x(\Phi(g)) = X(g)$ by~\eqref{E:XYPhi}, which leads to the expression on the right side of~(b).
\end{proof}

\section{Special \texorpdfstring{$L$}{L}-values} \label{S:Examples}

As in \cite{And96}, the identities in Corollary~\ref{C:logalg1} and Theorem~\ref{T:Main} can be used to recover information about $L(E_0/\QQ,1)$ or $L(E_0/\QQ,\chi,1)$, for a Dirichlet character $\chi$, when these $L$-values are non-vanishing, by specializing at certain values of $u$ and $t$ and for judicious choices of $\beta$. As mentioned in \S\ref{S:Intro}, the special value identities we obtain are closely related to those obtained previously (e.g., see \cite{Cremona}*{\S 2.8--2.12, App.\ Ex.~1}), but here our goal is to highlight how special $L$-value formulas can be reflected in formal power series identities.

We continue with the notation of the previous section, and let $\HH$ be the upper half-plane in $\CC$. We assume that the sign of the functional equation of $L(E_0/\QQ,s)$ is $\varepsilon = {+1}$. Recall that $\phi: X_0(N) \to E$ is the modular parametrization of $E$. Since $\varepsilon={+1}$, it is well-known using the Atkin-Lehner $w_N$ operator that for $\tau \in \HH$,
\begin{equation} \label{E:wN}
\phi(w_N(\tau)) = \phi \biggl( -\frac{1}{N\tau} \biggr) = [-1]_E(\phi(\tau)) + \phi(0),
\end{equation}
where the image $\phi(0)$ of the cusp $0$ is necessarily a torsion point on $E$ (e.g., see \cite{Birch75}*{\S 2}). Since $\tau = i/\sqrt{N}$ is fixed by $w_N$, we have
\begin{equation} \label{E:2Ptorsion}
[2]_E \circ \phi\biggl( \frac{i}{\sqrt{N}} \biggr) = \phi(0) \in E(\QQ)_{\tor}.
\end{equation}
By Corollary~\ref{C:logalg1}(b), we see that $\wp(\frac12 L(E_0/\QQ,1))$ represents the $x$-coordinate of $\phi(i/\sqrt{N})$ and thus a torsion point. Letting $d$ be its order, we see that
\begin{equation} \label{E:L1inlattice}
L(E_0/\QQ,1) \in \frac{2\Omega}{d} \cdot \ZZ,
\end{equation}
where $\Omega$ is the positive real period of $E$. We can pin down the exact value by computing the value of $L(E_0/\QQ,1)$ to enough precision so that its value can be determined from the fact that $2\Omega/d \cdot \ZZ$ is a discrete subset of $\RR$. See Example~\ref{Ex:one}.

In a similar manner, and by adapting the approach of~\cite{And96}, we can evaluate twists $L(E_0/\QQ,\chi,1)$ by applying Theorem~\ref{T:Main}. It is somewhat more complicated, relying on calculations with Heegner points, which we exhibit in Examples~\ref{Ex:two} and~\ref{Ex:three}. 

The computations of the exact values of $L(E_0/\QQ,1)$ and $L(E_0/\QQ,\chi,1)$ in Examples~\ref{Ex:one} and~\ref{Ex:two} are essentially the same as what can be found in \cite{Cremona}*{\S 2.8--2.11}. They are warm-ups for Example~\ref{Ex:three}, which requires additional considerations and where the point of view of Theorem~\ref{T:Main} is most useful. Calculations were performed using PARI~\cite{PARI}.

\begin{example} \label{Ex:one}
Let $E_0=X_0(11) : y^2+y=x^3-x^2-10x-20$ be the strong Weil curve of conductor $11$, with Hecke eigenform in $S_2(\Gamma_0(11))$,
\[
f = \prod_{m=1}^{\infty} ( 1 - q^m)^2(1-q^{11m})^2 = q- 2q^2 - q^3 + 2q^4 + q^5 + \cdots \rassign \sum_{n=1}^{\infty} a_n q^n.
\]
The sign of the functional equation of $L(E_0/\QQ,s) = L(f,s)$ is $\varepsilon = {+1}$. Applying the change of coordinates in~\eqref{E:coordchange}, we obtain
\[
E : y^2 = x^3 - \frac{31}{3} x - \frac{2501}{108}.
\]
As formal series, we have
\[
\wp(z) = \frac{1}{z^2} + \frac{31}{15}z^2+\frac{2501}{756}z^4 + \frac{961}{675}z^6 + \frac{77531}{41580}z^8 \cdots,
\]
and by Proposition~\ref{P:ExpE},
\[
\exp_{\cE}(z) = -\frac{2\wp(z)}{\wp'(z)} = z + \frac{62}{15}z^5 + \frac{2501}{252}z^7
+ \frac{1922}{135}z^9 + \cdots.
\]
We compute
\begin{align*}
X(t) &= \frac{1}{t^2} + \frac{2}{t} + \frac{11}{3} + 5t + 8t^2 + t^3 + 7t^4 + \cdots, \\
Y(t) &= -\frac{1}{t^3} - \frac{3}{t^2} - \frac{7}{t} - \frac{25}{2} -17t - 26t^2 -19t^3 + \cdots,
\end{align*}
and so
\[
\Phi(t) = -\frac{X(t)}{Y(t)} = t - t^2 - \frac{1}{3}t^3 + \frac{1}{2}t^4 + \frac{13}{3}t^5 - \frac{61}{3}t^6 + \frac{529}{12} t^7 + \cdots.
\]
We find that
\[
\Phi\bigl( e^{-2\pi/\sqrt{11}} \bigr) \approx 0.1270624598...,
\]
which corresponds to the point
\[
P = \bigl( X\bigl( e^{-2\pi/\sqrt{11}} \bigr), Y\bigl( e^{-2\pi/\sqrt{11}} \bigr) \bigr) \approx (62.111554..., -488.826947...). 
\]
We calculate $2P \approx (15.666666..., -60.499999...)$, and since by~\eqref{E:2Ptorsion},
\begin{equation} \label{E:Ex2Pphi0}
2P = \phi(0)  \in E(\QQ)_{\tor} = \{ O,\, (14/3,\, \pm 11/2),\, (47/3,\, \pm 121/2)\},
\end{equation}
we find
\begin{equation} \label{E:Ex2P}
2P = (47/3,\, -121/2).
\end{equation}
Since $P$ itself does not approximate any of the elements of $E(\QQ)_{\tor}$, we see that $P$ has order~$10$. By \eqref{E:Lrapid} we obtain $L(E_0/\QQ,1) \approx 0.2538418608...$,
and we can approximate that the minimal positive real period of $E$ is $\Omega \approx 1.2692093042...$. By \eqref{E:L1inlattice}, we see that $L(E_0/\QQ,1) \in (2\Omega/10)\ZZ$, and since this is a discrete set, by comparing approximations we find that
\begin{equation}
L(E_0/\QQ,1) = \frac{\Omega}{5}.
\end{equation}
\end{example}

\begin{example} \label{Ex:two}
We continue with the notation of Example~\ref{Ex:one}. Let $\chi : \ZZ \to \{0, \pm 1\}$ be the quadratic Dirichlet character for $\QQ(\sqrt{-3})$, and let
\[
L(E_0/\QQ,\chi,s) = \sum_{n=1}^{\infty} \chi(n) a_n n^{-s},
\]
be the corresponding quadratic twist of $L(E_0/\QQ,s)$. By~\cite{Cremona}*{Prop.~2.11.2},
\begin{equation} \label{E:LE0chi1}
L(E_0/\QQ,\chi,1) = 2 \sum_{n=1}^{\infty} \frac{\chi(n) a_n}{n} e^{-2\pi n/3\sqrt{11}},
\end{equation}
from which we obtain $L(E_0/\QQ,\chi,1) \approx 1.6844963329...$.
Now let $\beta = u-u^2$, and let $\rho = e^{2\pi i/3}$. It is a quick calculation to determine that
\[
\beta(\rho^n) = \rho^n - \rho^{-n} = \chi(n) \cdot \sqrt{-3}.
\]
We note that together with \eqref{E:LE0chi1} this implies,
\begin{equation}
\sum_{n=1}^{\infty} \frac{a_n \beta(u^n)}{n} t^n \biggm|_{u=\rho,\, t=\exp(-2\pi/3\sqrt{11})} = \frac{\sqrt{-3}}{2} L(E_0/\QQ,\chi,1).
\end{equation}
From Theorem~\ref{T:Main}(b), letting $P(t) = (X(t),Y(t)) \in E(\laurent{\QQ}{t})$ we see that
\[
\wp \Biggl( \sum_{n=1}^{\infty} \frac{a_n \beta(u^n)}{n} t^n \Biggr)
= x\bigl( P(ut) - P(u^2t) \bigr),
\]
where the difference on the right is calculated in $E(\laurent{\QQ(u)}{t})$. From this we obtain
\begin{equation} \label{E:wpLE0chi1}
\wp \biggl( \frac{\sqrt{-3}}{2} L(E_0/\QQ,\chi,1) \biggr)
= x\Bigl( P\bigl(\rho\cdot e^{-2\pi/3\sqrt{11}} \bigr) - P \bigl(\rho^{-1}\cdot e^{-2\pi/3\sqrt{11}}\bigr) \Bigr).
\end{equation}
Let
\[
Q \assign P \bigl(\rho \cdot e^{-2\pi/3\sqrt{11}} \bigr) \approx
(-2.055777... + i\cdot 1.071828...,
-0.336526... - i \cdot 1.905429...)
\]
and note $\oQ = P (\rho^{-1} \cdot e^{-2\pi/3\sqrt{11}})$,
where $\oQ$ is the complex conjugate of $Q$ in $E(\CC)$. Now
\[
\rho \cdot \exp \biggl( -\frac{2\pi}{3\sqrt{11}}\biggr)
= \exp \biggl( 2\pi i \biggl( \frac{1}{3} + \frac{\sqrt{-11}}{33} \biggr) \biggr),
\]
and so $\tau = 1/3 + \sqrt{-11}/33 \in \HH$ represents the Heegner point $Q \assign P(\rho \cdot e^{-2\pi/3\sqrt{11}}) \in E_0(\overline{\QQ})$. As $\tau$ is a root of $33x^2 - 22x + 4$, this is a Heegner point of discriminant $-44$. Moreover, in the notation of~\cite{Gross84}, 
\[
Q =  \phi \bigl( \ZZ[\sqrt{-11}], (\sqrt{-11}), [(11+\sqrt{-11},33)] \bigr),
\]
where $\ZZ[\sqrt{-11}]$ is the order of discriminant $-44$, $(\sqrt{-11}) \subseteq \ZZ[\sqrt{-11}]$ is an ideal of norm $11$, and $[(11+\sqrt{-11},33)]$ is an ideal class in the class group of $\ZZ[\sqrt{-11}]$.
What is important to note here is that by \cite{Gross84}*{Eq.~(5.2)}, $Q$ is fixed under the induced Atkin-Lehner involution on~$E$:
\[
w_{11}(Q) = \phi \bigl( \ZZ[\sqrt{-11}], (-\sqrt{-11}), \bigl[ (11+\sqrt{-11},33)\cdot (\sqrt{-11})^{-1} \bigr] \bigr) = Q.
\]
Therefore, as in~\eqref{E:wN}, \eqref{E:Ex2Pphi0}, and~\eqref{E:Ex2P}, we obtain
\[
2Q = \phi(0) = (47/3,\, -121/2).
\]
For $\oQ$ we start with $\tau = -1/3 + \sqrt{-11}/33$, and in the same manner find that $2\oQ=\phi(0)$ as well. Therefore, $Q-\oQ$ is a $2$-torsion point, and since $Q \neq \oQ$, its order is~$2$. By~\cite{SilvermanATAEC}*{Thm.~V.2.3} (and its proof), since $E(\RR)$ has a single component, we can write the period lattice of $E$ as $\ZZ\Omega + \ZZ\Omega'$, where $\Omega$ is the minimal positive real period and $\re(\Omega') = \frac12\Omega$, ($\im(\Omega') \approx -1.4588166169...$). From \eqref{E:wpLE0chi1}, as $(\sqrt{-3}/2)L(E_0/\QQ,\chi,1)$ is purely imaginary, we then have
\[
\frac{\sqrt{-3}}{2} L(E_0/\QQ,\chi,1) \in \frac{\Omega - 2\Omega'}{2} \cdot \ZZ,
\]
and by comparing approximations we find exactly
\begin{equation}
L(E_0/\QQ,\chi,1) = \frac{\Omega-2\Omega'}{\sqrt{-3}}.
\end{equation}

A few comments are in order.
(1) The specializations on the right-hand side of~\eqref{E:wpLE0chi1} converge in $\CC$, and the right-hand side is well-defined because $\wp$ is meromorphic on all of~$\CC$. However, the value $(\sqrt{-3}/2)L(E_0/\QQ,\chi,1)$ is outside of the radius of converge of the power series for $\wp(z)$ centered at $z=0$. On the other hand, $(\sqrt{-3}/2)L(E_0/\QQ,\chi,1) + \Omega'$ is within the radius of convergence, so by shifting the series for $\wp(z)$ by $\Omega'$, the identity in~\eqref{E:wpLE0chi1} holds.

(2) By taking the quadratic twist of $E_0$ by $-3$, we arrive at the strong Weil curve $F_0: y^2+y = x^3-3x-5$ of conductor $99$. Then $L(E_0/\QQ,\chi,s)= L(F_0/\QQ,s)$, and these same calculations can proceed as in Example~\ref{Ex:one}. Indeed $(\Omega - 2\Omega')/\sqrt{-3}$ turns out to be the real period associated to $F_0$. In Example~\ref{Ex:three} we consider the case of a cubic character where this method is not available and where we can generalize the present example.
\end{example}

\begin{example} \label{Ex:three}
We continue with the notation of the previous two examples, and now let~$\psi$ be the cubic Dirichlet character modulo $7$, satisfying for $\rho=e^{2\pi i/3}$,
\[
\psi(1)=1,\quad \psi(2)=\rho^2, \quad \psi(3)=\rho, \ldots.
\]
For $\zeta=e^{2\pi i/7}$, we let $g(\psi) = \sum_{j=1}^6 \psi(j)\zeta^j$ be the associated Gauss sum, and we set
\[
C_{\psi} \assign \psi(-11) \frac{g(\psi)}{g(\opsi)} = \rho \cdot \frac{g(\psi)}{g(\opsi)}, \quad 
C_{\opsi} \assign \opsi(-11) \frac{g(\opsi)}{g(\psi)} = \rho^2 \cdot \frac{g(\opsi)}{g(\psi)}
\]
Then $C_{\psi}$ is the sign of the functional equation of $L(f,\psi,s)$ by~\cite{Iwaniec}*{Thm.~7.6}, so that if $\Lambda(E_0/\QQ,\psi,s) = (7\sqrt{11}/2\pi)^s \Gamma(s) L(E_0/\QQ,\psi,s)$, then
\[
\Lambda(E_0/\QQ,\psi,s) = C_{\psi} \Lambda(E_0/\QQ,\opsi,2-s).
\]
In a similar manner to the derivation of~\eqref{E:LE0chi1} in~\cite{Cremona}*{Prop.~2.11.2}, one finds
\begin{equation} \label{E:LE0psi1}
L(E_0/\QQ,\psi,1) = S_{\psi} + C_{\psi} S_{\opsi},
\end{equation}
where
\begin{equation}
S_{\psi} \assign \sum_{n=1}^{\infty} \frac{\psi(n) a_n}{n}e^{-2\pi n/7\sqrt{11}}, \quad
S_{\opsi} \assign \sum_{n=1}^{\infty} \frac{\opsi(n) a_n}{n}e^{-2\pi n/7\sqrt{11}}.
\end{equation}
This leads to the approximation $L(E_0/\QQ,\psi,1) \approx 1.997106... + i \cdot 1.328439...$.
Setting $\gamma = \sum_{j=1}^6 \psi(j) u^j$ and $\ogamma = \sum_{j=1}^6 \opsi(j) u^j$, the theory of Gauss sums implies that
\[
\gamma(\zeta^k) = \opsi(k) g(\psi), \quad
\ogamma(\zeta^k) = \psi(k) g(\opsi).
\]
We form polynomials in $\ZZ[u]$,
\begin{align*}
\beta_1 &= \gamma + \ogamma = 2u - u^2 - u^3 - u^4 - u^5 + 2u^6, \\
\beta_2 &= \frac{1}{\sqrt{-3}}( \gamma - \ogamma ) = -u^2 +u^3 + u^4 - u^5,
\end{align*}
and by setting
\[
T_{i} \assign \sum_{n=1}^{\infty} \frac{a_n \beta_i(u^n)}{n} t^n \biggm|_{u=\zeta,\,t=\exp(-2\pi/7\sqrt{11})} \quad i=1,\,2,
\]
we find that
\[
T_1 = g(\opsi) S_{\psi} + g(\psi) S_{\opsi}, \quad
T_2 = -\frac{g(\opsi)}{\sqrt{-3}} S_{\psi} + \frac{g(\psi)}{\sqrt{-3}} S_{\opsi}.
\]
Both $T_1$ and $T_2$ are in $\RR$, and by comparing with~\eqref{E:LE0psi1}, we obtain
\begin{equation} \label{E:LE0psi1normalized}
(1-\sqrt{-3}) g(\opsi) L(E_0/\QQ,\psi,1) = T_1 - 3 T_2.
\end{equation}
With this in mind, we define
\begin{equation} \label{E:betadef}
\beta \assign \beta_1 - 3\beta_2 = 2u + 2u^2 - 4u^3 - 4u^4 + 2u^5 + 2u^6,
\end{equation}
and
\begin{equation} \label{E:cLdef}
T \assign T_1 - 3T_2 = \sum_{n=1}^{\infty} \frac{a_n \beta(u^n)}{n} t^n \biggm|_{u=\zeta,\,t=\exp(-2\pi/7\sqrt{11})} \quad \in \RR.
\end{equation}
Letting $P_k = P(\zeta^k \cdot e^{-2\pi/7\sqrt{11}}) \in E(\overline{\QQ})$, we see that
\[
P_k = \phi\biggl( \frac{k}{7} + \frac{\sqrt{-11}}{77} \biggr), \quad 1 \leq k \leq 6.
\]
For $k=1, \dots, 6$, the number $\tau_k = k/7 + \sqrt{-11}/77$ is a root of $539x^2 - 154kx + 11k^2+1$, and so each $\tau_k$ represents a Heegner point of discriminant $-2156$. In the notation of \cite{Gross84},
\[
P_k = \phi \bigl( \ZZ[7\sqrt{-11}], (11,7\sqrt{-11}), [(77k+7\sqrt{-11},539)] \bigr),
\]
where $\ZZ[7\sqrt{-11}]$ is the order of discriminant $-2156$, $\fn \assign (11,7\sqrt{-11})$ is an ideal in this order of norm~$11$, and $[\fa_k]$ represents the ideal class of $\fa_k \assign (77k+7\sqrt{-11},539)$. Through computations using PARI~\cite{PARI}, the ideal class $[\fa_1] \in \mathrm{Cl}(\ZZ[7\sqrt{-11}])$ has order~$8$, and furthermore,
\[
[\fa_1]^2 = [\fa_4], \quad [\fa_1]^3 = [\fa_2], \quad [\fa_1]^4=[\fn], \quad
[\fa_1]^5 = [\fa_5], \quad [\fa_1]^6 = [\fa_3], \quad [\fa_1]^7=[\fa_6].
\]
By \cite{Gross84}*{Eq.~(5.2)}, the induced action of the Atkin-Lehner operator on $P_k$ is
\[
w_{11}(P_k) = \phi\bigl(\ZZ[7\sqrt{-11}], \fn, [\fa_k \fn] \bigr),
\]
and so
\[
w_{11}(P_1) = P_5, \quad w_{11}(P_2) = P_6, \quad w_{11}(P_3) = P_4.
\]
By \eqref{E:betadef} and \eqref{E:cLdef}, when we apply Theorem~\ref{T:Main}(b) to $\beta$, we have as in Example~\ref{Ex:two},
\begin{align*}
\wp(T) &= x(2P_1 + 2P_2 - 4P_3 - 4P_4 + 2P_5 + 2P_6) \\
&= x\bigl( 2P_1 + 2w_{11}(P_1) + 2P_2 + 2w_{11}(P_2) - 4P_3 - 4w_{11}(P_3)\bigr),
\end{align*}
and so $(\wp(T),\tfrac12 \wp'(T)) = O$ by~\eqref{E:wN}. Therefore, by~\eqref{E:LE0psi1normalized} and the fact that $T \in \RR$, we see
\[
(1- \sqrt{-3}) g(\opsi) L(E_0/\QQ,\psi,1) \in \Omega \cdot \ZZ,
\]
and by comparing approximations, we find that the multiple on the left-hand side is~$10\Omega$. After some rearrangement we obtain the identity
\begin{equation}
L(E_0/\QQ,\psi,1) = \frac{5}{14} (1+\sqrt{-3}) g(\psi) \Omega,
\end{equation}
which aligns with \cite{Shimura77}*{Thm.~1} but provides the precise algebraic multiple of~$\Omega$.
\end{example}

With appropriate modifications to these examples, one could in principle determine the special value $L(E_0/\QQ,\psi,1)$ exactly for an arbitrary elliptic curve and Dirichlet character, e.g., using the techniques of \cite{And96}*{\S 4.7}. This would be interesting to carry out, but without more specific information about the particular curve and its Heegner points, these log-algebraic methods are limited and determining the values beyond what is already known qualitatively in~\cite{Shimura77}*{Thm.~1} would be difficult.

\end{document}